\documentclass[letterpaper,reqno,12pt]{amsart}
\usepackage[margin=1.2in]{geometry}

\usepackage{amssymb, enumerate}
\usepackage[all]{xy}
\usepackage{hyperref}
\hypersetup{colorlinks=true}

\makeatletter
\@namedef{subjclassname@2020}{%
	\textup{2020} Mathematics Subject Classification}
\makeatother

\theoremstyle{plain}
\newtheorem{thm}{Theorem}[section] 
\newtheorem{cor}[thm]{Corollary}
\newtheorem{prop}[thm]{Proposition}

\newtheorem{lem}[thm]{Lemma}

\newtheorem*{data availability}{Data Availability Statement}

\theoremstyle{definition} 
\newtheorem{defn}[thm]{Definition}
\newtheorem{notation}[thm]{Notation}

\theoremstyle{remark}
\newtheorem{rem}[thm]{Remark}
\newtheorem{ques}[thm]{Question}
\newtheorem*{acknowledgement}{Acknowledgments}

\def\phi{\varphi}
\def\epsilon{\varepsilon}

\def\mapsto{\longmapsto}

\newcommand{\sO}{\mathcal{O}}

\newcommand{\N}{\mathbb{N}}
\newcommand{\Q}{\mathbb{Q}} 
\newcommand{\C}{\mathbb{C}} 
\newcommand{\R}{\mathbb{R}} 
\newcommand{\Z}{\mathbb{Z}}

\newcommand{\NN}{{}^*\!\mathbb{N}}

\newcommand{\ba}{\mathfrak{a}}
\newcommand{\m}{\mathfrak{m}}
\newcommand{\n}{\mathfrak{n}}
\newcommand{\p}{\mathfrak{p}}

\newcommand{\cl}[2]{\operatorname{cl}_{\rm u}^{#2^t}(0)_{#1}}

\newsavebox{\circlebox}
\savebox{\circlebox}{\fontencoding{OMS}\selectfont\Large\char13}
\newlength{\circleboxwdht}

\def\Spec{\operatorname{Spec}}

\def\ulim{\operatorname{ulim}}

\def\Ker{\operatorname{Ker}}
\def\Ann{\operatorname{Ann}}
\def\Ass{\operatorname{Ass}}
\def\Char{\operatorname{char}}
\def\H{\operatorname{H}}
\def\UH{\operatorname{UH}}

\def\id{\operatorname{id}}

\title{A characterization of multiplier ideals via ultraproducts}

\author{Tatsuki Yamaguchi}
\address{Graduate School of Mathematical Sciences, University of Tokyo, 3-8-1 Komaba, Meguro-ku, Tokyo 153-8914, Japan}
\email{tyama@ms.u-tokyo.ac.jp}


\begin{document}

\begin{abstract}
In this paper, using ultra-Frobenii, we introduce a variant of Schoutens' non-standard tight closure \cite{affine}, ultra-tight closure, on ideals of a local domain $R$ essentially of finite type over $\C$. We prove that the ultra-test ideal $\tau_{\rm u}(R,\ba^t)$, the annihilator ideal of all ultra-tight closure relations of $R$, coincides with the multiplier ideal $\mathcal{J}(\Spec R,\ba^t)$ if $R$ is normal $\Q$-Gorenstein. As an application, we study a behavior of multiplier ideals under pure ring extensions.
\end{abstract}

\maketitle


\section{Introduction}
  The notion of tight closure, introduced by Hochster and Huneke \cite{HH}, is a powerful tool in commutative algebra of positive characteristic.
Tight closure gives simple proofs of several important results, e.g. Hochster-Roberts theorem, Brian\c{c}on-Skoda theorem, and so on, in equal characteristic. We say that a ring $R$ of positive characteristic is (weakly) $F$-regular if every ideal of $R$ is tightly closed.
The test ideal $\tau(R)$ of $R$, which defines the non-$F$-regular locus of $R$, is the annihilator ideal of all tight closure relations of $R$.
\ On the other hand, the multiplier ideal $\mathcal{J}(X)$ of a variety $X$ over a field of characteristic zero is formulated in terms of resolution of singularities \cite{La} and defines the non-log terminal locus of $X$. Hara and Yoshida \cite{HY} generalized the notions of tight closure and test ideals to the pair setting, and proved the following: if $R$ is a normal $\Q$-Gorenstein local ring essentially of finite type over $\C$, $\ba$ is a non-zero ideal of $R$ and $t$ is a positive real number, then modulo $p$ reductions of the multiplier ideal $\mathcal{J}(\Spec R,\ba^t)$ coincide with the test ideals $\tau(R_p,\ba_p^t)$ of modulo $p$ reductions $(R_p,\ba_p^t)$ of the pair $(R,\ba^t)$ for sufficiently large primes $p$.
In this paper, we give an ultraproduct-theoretic analog of their result.\\
\ The idea of ultraproducts goes back to the construction of a non-standard model of arithmetic by Skolem \cite{Sko} in 1934, and the notion of ultraproducts was introduced by \L o\'{s} \cite{Los} in 1955. Schoutens used ultraproducts to prove various results on commutative algebra. For example, he gave a new proof of Hochster-Roberts theorem for rings of  equal characteristic zero and gave a new construction of a balanced big Cohen-Macaulay algebra.
He introduced in \cite{affine} the notion of non-standard hulls, which enables us to apply some arguments in positive characteristic to rings of equal characteristic zero.  Schoutens \cite{log-terminal} used non-standard hulls to define ultra-$F$-regularity, an ultraproduct-theoretic analog of $F$-regularity. Based on \cite{Hara}, \cite{Mehta-Srinivas}, \cite{F-rat}, he showed that if $R$ is a normal $\Q$-Gorenstein local ring essentially of finite type over $\C$, then $\Spec R$ has log terminal singularities if and only if $R$ is ultra-$F$- regular. He also showed that if $R\hookrightarrow S$ is a pure ring extension of $\Q$-Gorenstein local domains essentially of finite type over $\C$ and if $\Spec S$ has log terminal singularities, then so does $\Spec R$. The aim of this paper is to generalize his results on log terminal singularities to the pair setting. In particular, we will characterize multiplier ideals via ultraproducts.\\

 Let $R$ be a local domain essentially of finite type over $\C$, $\ba$ be a non-zero ideal of $R$ and $t$ be a positive real number.
Using ultra-Frobenii, we introduce a variant of Schoutens' non-standard tight closure \cite{affine}, ultra-tight closure, on ideals (modules) of $R$.
The ultra-test ideal $\tau_{\rm u}(R,\ba^t)$ is defined to be the annihilator ideal of all ultra-tight closure relations of $R$. We can easily see that $\tau_{\rm u}(R,R^{1})=R$ if and only if $R$ is ultra-$F$-regular in the sense of Schoutens.
The following is our main result.
\begin{thm}[Corollary \ref{main1}]\label{1.1}
Suppose that $R$ is normal $\Q$-Gorenstein.
Then
\begin{equation}\tau_{\rm u}(R,\ba^t)=\mathcal{J}(\Spec R,\ba^t).
\end{equation}
\end{thm}
We briefly explain an idea of the proof of Theorem \ref{1.1}.
First, comparing modulo $p$ reductions of a multiplier ideal with its approximation, we see that test ideals of Hara and Yoshida give an approximation of the multiplier ideal.
The inclusion $\tau_{\rm u}(R,\ba^t)\supset\mathcal{J}(\Spec R,\ba^t)$ easily follows from this fact.
For the converse inclusion, we take the canonical cover of $R$ to reduce to the quasi-Gorenstein case and then use local ultra-cohomology.

As an application, of Theorem \ref{1.1}, we can show that multiplier ideals behave well under pure ring extensions. This result in characteristic zero is an analog of {\cite[Proposition1.12]{HY}}.
\begin{thm}[Corollary \ref{main2}]
Let $R, S$ be normal local domains essentially of finite type over $\C$, $\ba$ be a non-zero principal ideal of $R$ and $t$ be a positive real number. Assume that $R$ is $\Q$-Gorenstein and $R\hookrightarrow S$ is a pure ring extension. Then
\begin{equation}
\mathcal{J}(S,(\ba S)^t)\cap R\subseteq \mathcal{J}(R,\ba^t).
\end{equation}
Note that multiplier ideals can be defined without $\Q$-Gorenstein hypothesis (see Definition \ref{non-Q-Gorenstein}).
In particular, $\Spec S$ has Kawamata log terminal singularities in the sense of de Fernex-Hacon, then so does $\Spec R$.
\end{thm}

\begin{small}
\begin{acknowledgement}
The author wishes to express his gratitude to his supervisor Professor Shunsuke Takagi for his encouragement, valuable advice and suggestions. The author is also grateful to Tatsuro Kawakami, Kenta Sato and Shou Yoshikawa for useful conversations. Without their continuous help, this paper would not have been possible. He also thanks the referee who provided useful comments and suggestions.
\end{acknowledgement}
\end{small}
\section{Ultraproducts}
 In this section, we quickly review basic notions from the theory of ultraproduct. The reader is referred to \cite{affine}, \cite{use of ultraproducts} for details.
We fix an infinite set $W$. Let $\mathcal{P}(W)$ be the power set of $W$.
\begin{defn}
A nonempty subset $\mathcal{F} \subseteq \mathcal{P}(W)$ is called a {\it filter} if the following two conditions hold.
\begin{enumerate}
\item [$(\textup{i})$]If $A, B \in \mathcal{F}$, then $A \cap B\in \mathcal{F}$.
\item [$(\textup{ii})$]If $A \in \mathcal{F}, A \subseteq B \subseteq W$, then $B \in \mathcal{F}$.
\end{enumerate}
\end {defn}

\begin{defn}
\item [$(\textup{i})$]A filter $\mathcal{F}$ is called an {\it ultrafilter} if for all $A \in \mathcal{P}(W)$, $A \in \mathcal{F}$ or $A^c \in \mathcal{F}$.
\item [$(\textup{ii})$]A filter$\mathcal{F}$ is called {\it principal} if there exists a finite subset $A\subseteq W$ such that $A \in \mathcal{F}$.
\end{defn}
 From now on, we only consider a non-principal ultrafilter on $W$. Let $A_w$ be a family of sets indexed by $W$. We can define an equivalence relation $\sim$ on $\prod A_w $ by
 \begin{equation}
 (a_w)\sim(b_w) \text{ if and only if } \{w|a_w=b_w\} \in \mathcal{F}.
 \end{equation}
\begin{defn}
Let $A_w$ be a family of sets indexed by $W$ and $\mathcal{F}$ be a non-principal ultrafilter on $W$. The {\it ultraproduct} of $(A_w)$ is defined by 
\begin{equation}
\ulim A_w = A_{\infty} := \prod A_w/\sim.
\end{equation}
 We denote the equivalence class of $(a_w)$ by $\ulim a_w$.
 \end{defn}
 For a family of maps $f_w:A_w\to B_w$ indexed by $W$, we can also define the ultraproduct of maps $f_\infty :A_\infty \to B_\infty$ by
 \begin{equation}
 f_\infty (\ulim a_w):= \ulim f_w(a_w).
 \end{equation}
 If each $A_w$ has an operation $*_w:A_w\times A_w\to A_w$, then we can define the natural operator $*_\infty$on $A_\infty$ by
 \begin{equation}
 (\ulim a_w) *_\infty (\ulim b_w):=\ulim (a_w *_w b_w).
 \end{equation}
 For example, if each $A_w$ is a (commutative) ring, we can prove $A_\infty$ has the natural (commutative) ring structure.
\begin{defn}
Let $\mathcal{P}$ be a some property. We say that $a_w$ satisfies $\mathcal{P}$ for {\it almost all} $w$ if $\{w|\mathcal{P}(a_w)\text{ holds}\}\in\mathcal{F}$.
\end{defn}
An important property of ultraproducts is stated in \L o\'{s} theorem. We will state some weaker version of \L o\'{s} theorem that Schoutens stated in his paper.
\begin{thm}[{\cite[Theorem 2.3]{affine}}]
Suppose that $C$ is a ring and $A_\infty$ is the ultraproduct of $C$-algebras $A_w$. Let $f$ be a polynomial in $n$ variables with coefficients in $C$ and for each $w$ we take an $n$-tuple ${\bf a}_w$ of elements of $A_w$. Then $f(\ulim {\bf a}_w)=0 $ in $ A_\infty $ if and only if $ f({\bf a}_w)=0$ in $A_w$ for almost all w.
\end{thm} 
 At the end of this section, we will state the important theorems about ultraproducts.
\begin{thm}[{\cite[2.8.2]{affine}}]
If almost all $K_w$ are algebraically closed field, then $K_\infty$ is algebraically closed.
\end{thm}
\begin{thm}[Lefschetz principle, {\cite[Theorem 2.4]{affine}}]
Let $W$ be the set of prime numbers endowed with some non-principal ultrafilter. Then
\begin{equation}
\ulim \mathbb{F}_p^{\text{alg}}\cong \mathbb{C}.
\end{equation}
\end{thm}
\section {Non-standard hulls and approximations}
In this section, we consider the set $\mathcal{P}$ of prime numbers endowed with some non-principal ultrafilter as an infinite set $W$ in the section 2. Suppose that $B$ is an algebra of finite type over $\mathbb{C}$. Then we can define the non-standard hull $B_\infty$ of $B$ and an approximation $B_p$ of $B$ as in \cite[3.4]{affine}. If $\p$ is a prime ideal of $B$, then $\p B_\infty$ is a prime ideal of $B_\infty$.
\begin{defn}[{\cite[4.3]{affine}}]
Suppose that $B$ is an algebra of finite type over $\C$, $\p$ is a prime ideal of $B$ and $R$ is a localization $B_\p$ of $B$.
Then the {\it non-standard hull} $R_\infty$ of $R$ is defined by
\begin{equation}
R_\infty=(B_\infty)_{\p B_\infty}.
\end{equation}
An {\it approximation} $R_p$ is defined as in \cite[4.3]{affine}. Then $R_\infty=\ulim R_p$.
\end{defn}
\begin{prop}[{\cite[4.3]{affine}}] Let the notation be as above. We list basic properties of the non-standard hull of the local algebra $R$.
\begin{enumerate}
 \item[$(\textup{i})$] $R_\infty =\ulim R_p$, and $R\hookrightarrow R_\infty$ is a faithfully flat extension.
 \item[$(\textup{ii})$] $R_p$ are local rings essentially of finite type over $\mathbb{F}_p^{\text{alg}}$ for almost all $p$.
\end{enumerate}
\end{prop}
\begin{defn}
\begin{enumerate}
 \item[$(\textup{i})$]For $f \in R$, we can write $f=\ulim f_p$ through $R\hookrightarrow R_\infty$. We call the collection of $f_p$ an {\it approximation} of $f$.
 \item[$(\textup{ii})$]For an ideal $\mathfrak{a}\subseteq R$ such that $\mathfrak{a}=(a_1, \dots ,a_n)$, the collection of $\mathfrak{a}_p:=(a_{1p},\dots, a_{np})$ is called an {\it approximation} of the ideal $\mathfrak{a}$. Then $\ulim \mathfrak{a}_p=\mathfrak{a} R_\infty$.
\end{enumerate}
\end{defn}
\begin{prop}[{\cite[Theorem 4.6]{affine}}]
 Let $R$ be a local ring essentially of finite type over $\C$ with an approximation $R_p$. If $R$ is regular (resp. complete intersection, Gorenstein, Cohen-Macaulay), then so is $R_p$ for almost all $p$. 
\end{prop}
\begin{defn}[Models]
Let $K$ be a field and let $B$ be an algebra of finite type over a field $K$.
A pair $(Z, B_Z)$ is called a {\it model} of $B$ if $Z$ is a finitely generated $\mathbb{Z}$-subalgebra of $K$, and $B_Z$ is a finitely generated $Z$-algebra such that $B\cong B_Z\otimes_Z K$.
\end{defn}
\begin{defn}
Let $B$ be an algebra of finite type over $\C$ and $(Z,B_Z)$ be a model of $B$.
Suppose that $\mu$ is a maximal ideal of $Z$. Then $B_Z\otimes_Z \kappa(\mu)$ is said to be a {\it reduction to characteristic $p$} if $\Char \kappa(\mu)=p$.
\end{defn}
\begin{defn}
Let $B$ be an algebra of finite type over $\C$, $(Z,B_Z)$ be a model of $B$ and $\ba$ be an ideal of $B$.
Suppose that $\ba_Z$ is an ideal of $B_Z$ such that $\ba_Z B=\ba$. Then the ideal $\ba_Z(B_Z\otimes_Z\kappa(\mu))$ is said to be a {\it reduction} of the ideal $\ba$.
\end{defn}
An important relation between approximations and models is the following;
\begin{thm}[{\cite[Proposition 2.18]{log-terminal}}]\label{affine model}
Let $R$ be an algebra of finite type over $\mathbb{C}$. We can find a model $(Z, R_Z)$ of $R$ with the following property: for almost all $p$ there exists a maximal ideal $\p_p$ such that the collection of base changes $R_Z\otimes_Z \mathbb{F}_p^\text{alg}$ gives an approximation of $R$.
\end{thm}

\begin{prop}
Let $R$ be a local domain essentially of finite type over $\mathbb{C}$. The  following conditions are equivalent to each other.
\begin{enumerate}
\item [$(\textup{i})$] $R$ is normal.
\item [$(\textup{ii})$] $R_p$ is normal for almost all $p$.
\item [$(\textup{iii})$] $R_\infty$ is normal.
\end{enumerate}
\end{prop}
\begin{proof}

	The condition (iii) follows from (ii) by \L o\'{s}' theorem. 
	The condition (i) follows from (iii) since $R \hookrightarrow R_\infty$ is faithfully flat.
	It is enough to show that the condition (ii) follows from (i). Suppose that $S$ is a normal domain of finite type over $\mathbb{C}$ such that there exists a prime ideal $\mathfrak{p}$ of $S$ such that $R\cong S_\mathfrak{p}$.
	Then for almost all $p$ , $R_p \cong (S_p)_{\mathfrak{p}_p}$. Let $(Z, S_Z)$ be a model of $S$ as in Theorem \ref{affine model}. A set 
	\begin{equation*}
			\mathcal{N}:=\{\mu\in \Spec Z| S_Z\otimes_Z \kappa (\mu) \text{ is geometrically normal over }\kappa(\mu)\}
		\end{equation*}
	is constructible by \cite[Corollaire (9.9.5)]{EGAIV} and contains the generic point of $\Spec Z$. Hence $\mathcal{N}$ contains a dense open subset of $\Spec Z$. Enlarging the model $S_Z$, we may assume that $S_Z\otimes_Z\kappa(\mu)$ is geometrically normal over $\kappa(\mu)$ for all $\mu\in \Spec(Z)$. Therefore, $S_p$ is normal for almost all $p$. Hence, $R_p$ is normal for almost all $p$.
\end{proof}
\begin{defn}
Let $R$ be a Noetherian ring, $I$ be an ideal of $R$ and $n$ be a natural number. Then the {\it $n$-th symbolic power} of $I$ is the ideal
\begin{equation}
I^{(n)}=R\cap\bigcap_{\p\in\Ass (R/I)}I^nR_\p,
\end{equation}
where $\p$ runs through all the associated primes of $R/I$.
\end{defn}
\begin{prop}
Let $R$ be a local ring essentially of finite type over $\mathbb{C}$ and let $I$ be an ideal of $R$. Fix a natural number $n$, taking symbolic power commutes with taking approximations that is $(I_p)^{(n)}=(I^{(n)})_p$ holds for almost all $p$.
\end{prop}
\begin{proof}
It is enough to show the case $I=\p$ is a prime ideal.
Since $\mathfrak{p}^n\subseteq \mathfrak{p}^{(n)} $and $\mathfrak{p}^{(n)}$ is $\mathfrak{p}$-primary, $(\mathfrak{p}_p)^{n}\subseteq(\mathfrak{p}^{(n)})_p$ and $(\mathfrak{p}^{(n)})_p$ is $\mathfrak{p}_p$-primary. Therefore $(\mathfrak{p}_p)^{(n)}\subseteq(\mathfrak{p}^{(n)})_p$ holds.
On the other hand, $\mathfrak{p}^n (R_\mathfrak{p})_\infty \cap R =\mathfrak{p}^{(n)}$ and $(\mathfrak{p}_p)^{n}(R_\mathfrak{p})_p\cap R_p=(\mathfrak{p}_p)^{(n)}$. Therefore, 
\begin{equation}
\mathfrak{p}_\infty ^{(n)} \cap R = \mathfrak{p}^{(n)}
\end{equation} and 
\begin{equation}
((\mathfrak{p}_\infty)^{(n)})_p=(\mathfrak{p}_p)^{(n)}.
\end{equation}
Thus $\mathfrak{p}^{(n)} R_\infty \subseteq (\mathfrak{p}_\infty)^{(n)}$ and $(\mathfrak{p}^{(n)})_p\subseteq ((\mathfrak{p}_\infty)^{(n)})_p=(\mathfrak{p}_p)^{(n)}$.
\end{proof}
\begin{cor}
Suppose that $R$ is a $\Q$-Gorenstein local domain essentially of finite type over $\C$. Then an approximation $R_p$ of $R$ is $\Q$-Gorenstein for almost all $p$.
\end{cor}

\section {Multiplier ideals and test ideals}
 In this section, we quickly review test ideals of Hara and Yoshida, and multiplier ideals.
\begin{defn}[\cite{HY}]
Let $R$ be a Noetherian domain of characteristic $p>0$, $\ba$ be a non-zero ideal and $t$ be a positive real number. For an ideal $I$ of $R$, we define the $\ba^t${\it -tight closure} $I^{*\ba^t}$ of $I$ as follows: $z\in I^{*\ba^t}$ if and only if there exist $c\in R\setminus \{0\}$ such that $cz^q\ba^{\lceil tq\rceil}\subseteq I^{[q]}$ for all $q=p^e\gg 0$.
\end{defn}
\begin{defn}[\cite{HY}]
Let $R$ be a Noetherian domain of characteristic $p>0$, $\ba$ be a non-zero ideal and $t$ be a positive real number. Suppose that $M$ is an $R$-module and $F:R\to R$ is the Frobenius morphism. Then $z$ lies in the $\ba^t$-{\it tight closure} $0^{*\ba^t}_M$ of 0 in $M$ if there exists an element $c \in R\setminus \{0\}$ such that $ca\otimes z =0 \in F^e_*R\otimes M$ for all $a\in\ba^{\lceil tq \rceil}$, for all $q=p^e\gg 0$.
\end{defn}
\begin{defn}[\cite{HY}]
Let $(R,\m)$ be an excellent local domain of characteristic $p>0$, $E=E_R(R/\m)$ be the injective hull of the residue field $R/\m$, $\ba$ be a non-zero ideal and $t$ be a positive real number. Then we define the {\it test ideal} $\tau(R,\ba^t)$ of $(R,\ba^t)$ by
\begin{equation}
\tau(R,\ba^t)=\bigcap_{M\subseteq E} \Ann_R(0_M^{*\ba^t}),
\end{equation}
where $M$ runs through all finitely generated $R$-submodules of $E$.
\end{defn}
\begin{defn}[{\cite[Definition 9.3.60]{La}}]
Let $X$ be a normal $\Q$-Gorenstein variety over $\C$, $\ba$ be a non-zero ideal sheaf of $\sO_X$ and $t$ be a positive real number.
Take a log resolution $\mu:Y\to X$ of $\ba$ with $\ba \sO_Y=\sO_Y(-Z)$. Then the {\it multiplier ideal} $\mathcal{J}(X,\ba^t)$ of a pair $(X,\ba^t)$ is defined by
\begin{equation}
\mathcal{J}(X,\ba^t)=\mu_*\sO_Y(\lceil K_Y-\mu^*K_X-tZ\rceil).
\end{equation}
If $X=\Spec R$, we also denote $\mathcal{J}(\Spec R,\ba^t)$ by $\mathcal{J}(R,\ba^t)$.
\end{defn}
\begin{defn}
Let $X$ be a normal $\Q$-Gorenstein variety over $\C$, $\ba$ be a non-zero ideal sheaf of $\sO_X$ and $t$ be a positive real number.
$(X,\ba^t)$ is said to be {\it Kawamata log terminal} if the multiplier ideal $\mathcal{J}(X,\ba^t)$ is trivial.
If $X=\Spec R$, then $R$ is said to be {\it Kawamata log terminal}. 
\end{defn}
\begin{thm}\label{approx of mult}
 Let $R$ be a $\Q$-Gorenstein local domain essentially of finite type over $\mathbb{C}$ and let $\mathfrak{a}$ be a non-zero ideal of $R$ and $t$ be a positive real number.
 Then \begin{equation}
 \mathcal{J}(R, \mathfrak{a}^t)_p=\tau (R_p, \mathfrak{a}_p^t)
\end{equation} for almost all $p$.
\end{thm}
\begin{proof}
Suppose that $S$ is a $\Q$-Gorenstein finitely generated $\mathbb{C}$-domain and $\mathfrak{p}$ is a prime ideal of $S$ such that $R\cong S_\mathfrak{p}$.
Let $\mathfrak{b}$ be $\mathfrak{a}\cap S$ and let $(Z, S_Z)$ be a model of S, and $\mu$ be a maximal ideal of $Z$ and $S_\mu$ be a reduction to characteristic $p>0$, $S_Z\otimes _Z \kappa (\mu)$.
Let $\mathfrak{b}_\mu$ be the reduction of $\mathfrak{b}$.
Then $R_p$ is a localization of $S_\mu\otimes \mathbb{F}_p^{\text{alg}}$ for some $\mu$.
Since multiplier ideals commute with localizations and test ideals also commute with localizations,
\begin{equation}
\mathcal{J}(R,\mathfrak{a}^t)=\mathcal{J}(S, \mathfrak{b}^t)R, \tau (R_p, \mathfrak{a}_p^t)=\tau (S_\mu\otimes \mathbb{F}_p^{\text{alg}}, (\mathfrak{b}_\mu\otimes\mathbb{F}_p^{\text{alg}})^t)R_p
\end{equation}
By arguments similar to \cite[Lemma 1.5]{Srni-Tak}. 
\begin{equation}
\tau (S_\mu\otimes \mathbb{F}_p^{\text{alg}}, (\mathfrak{b}_\mu\otimes\mathbb{F}_p^{\text{alg}})^t)=\tau (S_\mu, \mathfrak{b}_\mu^t)\otimes\mathbb{F}_p^{\text{alg}}
\end{equation}
By \cite[Theorem 6.8]{HY}, reduction to characteristic $p \gg0$, $\mathcal{J}(S,\mathfrak{b}^t)_\mu=\tau (S_\mu, \mathfrak{b}_\mu^t)$.
In conclusion, $\tau (R_p, \mathfrak{a}_p^t)$ gives an approximation of $\mathcal{J}(R,\mathfrak{a}^t)$.
 \end{proof}
 
\section{ultra-test ideal}
In this section, we will consider a pair $(R, \mathfrak{a}^t)$ where $R$ is a normal $\mathbb{Q}$-Gorenstein local domain essentially of finite type over $\mathbb{C}$, $\mathfrak{a}$ is a non-zero ideal of $R$ and  $t\in \R_{>0}$.
\begin{notation}
\begin{enumerate}
\item[(\textup{i})]We denote $\ulim \N$ the ultraproduct of countably many copies of $\N$ by $\NN$. This is called the set of  non-standard natural number in the non-standard analysis.
\item[(\textup{ii})] For  $\epsilon=\ulim e_p \in \NN$ and $t\in \R_{>0}$, we denote $\ulim \lceil te_p\rceil$ by $\lceil t \epsilon\rceil$. 
\item[(\textup{iii})] For $\epsilon=\ulim e_p \in \NN$, we denote $\ulim p^{e_p}$ by $\pi ^ \epsilon$.
\item[(\textup{iv})] For $\epsilon=\ulim e_p \in \NN$ and an ideal $\mathfrak{a}$ of $R$, we define a non-standard power of the ideal $\ba^\epsilon = \ulim \ba_p^{e_p}\subseteq R_\infty$.
\end{enumerate}
\end{notation}
\begin{defn}
$(R,\ba^t)$ is said to be a {\it pair} if $R$ is a local domain essentially of finite type over $\mathbb{C}$, $\mathfrak{a}$ is a non-zero ideal of $R$ and $t$ is a positive real number.
A pair $(R,\ba^t)$ is said to be a {\it $\Q$-Gorenstein pair} if $R$ is normal $\Q$-Gorenstein.
A pair $(R,\ba^t)$ is said to be a {\it quasi-Gorenstein pair} if $R$ is normal quasi-Gorenstein.
\end{defn}
\begin{defn}
For $\epsilon=\ulim e_p \in\NN$, we define an {\it ultra-Frobenius} $F^\epsilon :R\to R_\infty$ as the map sending an element $x$ of $R$ to an element $\ulim x_p^{p^{e_p}}$. We denote $R_\infty$ viewed as an $R$-module via $F^\epsilon$ by $F^\epsilon_*R_\infty $.
\end{defn}
\begin{defn}[ultra-tight closure of zero]
Let $(R,\ba^t)$ be a pair.
We define the {\it ultra-tight closure} $\cl{M}{\ba}$ {\it of zero} by
\begin{equation}
\cl{M}{\ba}:=\bigcup_{c\in R^{\circ}}\bigcap_{\epsilon\in\NN}\bigcap_{a\in\ba^{\lceil t\pi^{\epsilon}\rceil}}\Ker(M\to F^\epsilon R_\infty\otimes_R M; \eta\mapsto ca\otimes \eta)
\end{equation}
where $R^{\circ}$ is $R\setminus \{0\}$.
\end{defn}
\begin{defn}[ultra-test ideal]
Let $(R,\ba^t)$ be a pair and $\m$ be the maximal ideal of $R$. We define the {\it ultra-test ideal} $\tau_{\rm u}(R,\ba^t)$ of a pair $(R,\ba^t)$ by
\begin{equation}
\tau_{\rm u} (R,\ba^t):=\Ann_R \cl{E}{\ba} 
\end{equation}
where $E=E_R(R/\m)$ is the injective hull of the residue field $R/\m$.
\end{defn}
\begin{defn}[Ultra-$F$-regular]
Let $(R,\ba^t)$ be a $\Q$-Gorenstein pair.
$(R,\ba^t)$ is said to be {\it ultra-$F$-regular} if $\tau_{\rm u} (R,\ba^t)=R$.
\end{defn}
\begin{defn}
A ring extension $R\hookrightarrow S$ is said to be {\it pure} if for any $R$-module $M$, $M\to S\otimes_R M$ is injective.
\end{defn}
\begin{lem} The following are equivalent.
\begin{enumerate}
\item[$(\textup{i})$] $(R,\ba^t)$ is an ultra-$F$-regular pair.
\item[$(\textup{ii})$]For all $c\in R^{\circ}$, there exist $\epsilon\in\NN$ and $a\in\ba^{\lceil t\pi^{\epsilon}\rceil}$ such that 
\begin{equation}
E\to F^\epsilon R_\infty\otimes_R E; \eta\mapsto ca\otimes \eta
\end{equation}
is injective.
\item[$(\textup{iii})$]For all $c\in R^{\circ}$, there exist $\epsilon\in\NN$ and $a\in\ba^{\lceil t\pi^{\epsilon}\rceil}$ such that 
\begin{equation}
R\to F^\epsilon R_\infty;x\mapsto caF^\epsilon(x)
\end{equation} 
is pure.
\end{enumerate}
\end{lem}
\begin{proof}
The condition that $(R,\ba^t)$ is an ultra-$F$-regular pair is equivalent to the following by definition.
\begin{equation}
0=\bigcup_{c\in R^{\circ}}\bigcap_{\epsilon\in\NN}\bigcap_{a\in\ba^{\lceil t\pi^{\epsilon}\rceil}}\Ker(E\to F^\epsilon R_\infty\otimes_R E; \eta\mapsto ca\otimes \eta)
\end{equation}
Hence, the second condition is equivalent to the first. Let $M,N$ be $R$-modules and $\varphi : M\to N$ be an $R$-module homomorphism. Then, $\varphi:M \to N$ is pure if and only if $\varphi\otimes E:M\otimes E\to N\otimes E$ is injective. Therefore, the third condition is equivalent to the second.
\end{proof}
\begin{cor}
Let $R$ be a local domain essentially of finite type over $\C$. The following are equivalent.
\begin{enumerate}
\item[$(\textup{i})$] $(R,R^t)$ is an ultra-F-regular pair.
\item[$(\textup{ii})$] For all $c\in R^\circ$, there exists $\epsilon\in\NN$ such that $R\to F^\epsilon R_\infty;x\mapsto cF^\epsilon(x)$
\end{enumerate}
The above conditions are equivalent to the ultra-$F$-regularity of $R$ in \cite[Definition 3.3]{log-terminal}.
\end{cor}
\begin{proof}
This follows from the above lemma.
\end{proof}

\begin{defn}[finitistic ultra-test ideal]
Let $(R,\ba^t)$ be a $\Q$-Gorenstein pair.
\begin{equation}
\tau_{\rm fu}(R,\ba^t):=\bigcap_{M}\Ann_R\cl{M}{\ba},
\end{equation}
where $M$ runs through all finitely generated $R$-modules.
\end{defn}
\begin{prop}\label{M<E}
Let $(R,\ba^t)$ be a $\Q$-Gorenstein pair.Then 
\begin{equation}
\tau_{\rm fu}(R,\ba^t)=\bigcap_{M}\Ann_R\cl{M}{\ba}, 
\end{equation}
where $M$ runs through all finitely generated $R$-submodules of $E$.
\end{prop}
\begin{proof}
$\tau_{\rm fu}(R,\ba^t)\subseteq \bigcap_{M\subset E}\Ann_R\cl{M}{\ba}$, where $M$ runs through all finitely generated $R$-submodules of $E$, is clear. We need to prove the reverse inclusion. If 
\begin{equation}
a\in \bigcap_{M \subset E}\Ann_R\cl{M}{\ba},
\end{equation}
 where $M$ runs through all finitely generated $R$-submodules of $E$, but $a\notin\tau_{\rm fu}(R,\ba^t)$,
then there exist a finitely generated $R$-module such that $a\notin \Ann_R\cl{M}{\ba}$.
Hence there exists $x\in \cl{M}{\ba}$ such that $ax\neq 0$. Then there exists $n$ such that $ax\notin\m^n M$. Let 
\begin{equation}
N:=M/\m^nM, \overline{x}:=x+\m^nM,
\end{equation}
then $\overline{x}\in\cl{N}{\ba}$. Since $\Ass N=\{\m\}$, there exists $l\in \N$ such that $N\hookrightarrow E^{\oplus l}$. Therefore, $a\overline{x}=0$, and this is a contradiction.
\end{proof}
\begin{cor}
Let $(R,\ba^t)$ be a $\Q$-Gorenstein pair. Then
$\tau_{\rm u}(R,\ba^t)\subseteq \tau_{\rm fu}(R,\ba^t)$
\end{cor}
\begin{proof}
If $M \subset E$, then $\cl{M}{\ba}\subset \cl{E}{\ba}$. Hence, the conclusion follows. 
\end{proof}
\begin{prop}\label{fg}
Let $(R,\ba^t)$ be a $\Q$-Gorenstein pair and let $J$ be an ideal of $R$ which has finite order in the divisor class group $\operatorname{Cl}(R)$.
Then $
\cl{{\operatorname{H}_{\m}^d(J)}}{\ba}=\operatorname{cl} _{\operatorname{H}_\m^d (J)}^{\ba^t \text{fg}}(0)
$, where 
\begin{equation}
\operatorname{cl} _{\operatorname{H}_\m^d (J)}^{\ba^t \text{\rm fg}}(0)=\bigcup_{M}\cl{M}{\ba},
\end{equation}.
where $M$ runs through all finitely generated $R$-submodules of $E$.
\end{prop}
\begin{proof}
Take $x_1\in R$ such that $J^{(r)}=(x_1)$ and let $U=R\setminus \bigcup_{\p\in\Ass R/J}\p$. Then there exists an element $a\in J$ such that $JU^{-1}R=aU^{-1}R$ and there exists $x_2 \in U$ such that $x_2 J\subseteq (a)R$. Take $x_1,\dots, x_d$ to be a system of parameters of $R$. Then $x_{1p},\dots, x_{dp}$ has the same property for $J_p$ and $U_p$. Let $H$ be $\operatorname{H}_{\m}^d(J)$ and let $H_p$ be $\operatorname{H}_{\m_p}^d(J_p)$. 
Then $H\cong \varinjlim_s R/(x_1^sJ,x_2^s,\dots,x_d^s)$ and $H_p\cong \varinjlim_s R_p/(x_{1p}^sJ_p,x_{2p}^s,\dots,x_{dp}^s)$.
For $\eta\in H$,
\begin{equation} 
\eta=[z+(x_1^sJ,\dots,x_d^s)],
\end{equation}
 we define the approximation of $\eta$ by 
\begin{equation}
\eta_p=[z_p+(x_{1p}^sJ_p,\dots, x_{dp}^s)].
\end{equation}
If $\eta\in \cl{H}{\ba}$ such that $\eta=[z+(x_1^sJ,\dots,x_d^s)]$, using the proof by contradiction, $\eta_p \in 0_{H_p}^{*\ba^t}$ for almost all $p$.
$\eta_p=[z_p+(x_{1p}^sJ_p,\dots, x_{dp}^s)]$, then $[z_p] \in 0_{R_p/(x_{1p}^{s+1}J_p,x_{2p}^{s+1},\dots,x_{dp}^{s+1})}^{*\ba^t}$ by the proof of \cite[Theorem 1.13, Definition-Theorem 6.5]{HY}. Hence $\eta \in \cl{{R/(x_1^{s+1}J,x_2^{s+1},\dots,x_d^{s+1})}}{\ba}$ and $\eta \in \operatorname{cl}_H^{\ba^t\text{fg}}(0)$.
\end{proof}
\begin{cor}\label{finitistic ultra}
Let $(R,\ba^t)$ be a $\Q$-Gorenstein pair.
 Then 
\begin{equation}
\tau_{\rm u}(R,\ba^t)=\bigcap_{M}\Ann_R\cl{M}{\ba}=\tau_{\rm fu}(R,\ba^t),
\end{equation}
 where $M$ runs through all finitely generated $R$-submodules of $E$.
\end{cor}
\begin{proof}
By Proposition \ref{M<E} and Proposition \ref{fg}, the conclusion follows.
\end{proof}
\begin{prop}\label{mult<test}
	Let $(R,\ba^t)$ be a $\Q$-Gorenstein pair. Then $\mathcal{J}(R, \ba^t)\subseteq \tau_{\rm u}(R,\mathfrak{a}^t)$.
\end{prop}
\begin{proof}
	Let $x\in \mathcal{J}(R,\ba^t)=\ulim \tau(R_p, \ba_p^t)\cap R$.
	Let $\eta \in \cl{E}{\ba}$, then there exist
	\begin{equation}
		c\in \mathcal{J}(R,\ba^t)\cap R^\circ
	\end{equation}
	such that for any $\epsilon \in \NN$, $a\in\ba^{\lceil t\pi^{\epsilon}\rceil}$, the image of $\eta$ under the morphism 
	\begin{equation}
		E\to F^\epsilon_* R_\infty\otimes_R E; \eta\mapsto ca\otimes \eta
	\end{equation}
	is zero.
	By Proposition \ref{fg}, there exists a finitely generated submodule $M$ of $E$ which contains $\eta$ such that the image of $\eta$ under the morphism 
	\begin{equation}
		M\to F^\epsilon_* R_\infty\otimes_R M; \eta\mapsto ca\otimes \eta
	\end{equation}
	is zero.
	If for almost all $p$, $\eta_p\notin 0_{M_p}^{*\ba_p^t}$, then there exist  $e_p \in \N$ and $a_p \in \ba_p^{\lceil tp^{e_p}\rceil}$, the image of $\eta_p$ under the morphism 
	\begin{equation}
		M_p\to F^{e_p} _* R_p\otimes_{R_p} M_p; \eta_p\mapsto ca\otimes \eta_p
	\end{equation}
	is non-zero.
	Then we take $\eta:=\ulim e_p, a:=\ulim a_p$, then taking the ultraproduct, the image of $\eta$ under the morphism 
	\begin{equation}
		M\to F^\epsilon_* R_\infty\otimes_R M; \eta\mapsto ca\otimes \eta
	\end{equation}
	is non-zero. This is a contradiction. Therefore $\eta_p\in 0_{M_p}^{*\ba_p^t}$ for almost all $p$.
	Then $x_p\eta_p=0$ for almost all $p$ and $x\eta=0$. In conclusion, $x\in \tau_{\rm u}(R, \ba^t)$.
\end{proof}
\begin{cor}
	If $(R,\ba^t)$ is a Kawamata log terminal pair, then $(R, \ba^t)$ is ultra-$F$-regular.
\end{cor}
\begin{lem}\label{ideal finite extension}
Let $R,S$ be local rings essentially of finite type over $\C$ such that $S$ is a finite extension of $R$, and let $\ba$ be an ideal of $R$.
Then $\ba^{\lceil t\pi ^\epsilon\rceil}S_\infty=(\ba S)^{\lceil t\pi ^\epsilon\rceil}$.
\end{lem}
\begin{proof}
$\ba^{\lceil t\pi ^\epsilon\rceil}S_\infty\subseteq (\ba S)^{\lceil t\pi ^\epsilon\rceil}$ is clear. To prove the reverse inclusion, let $\zeta_1,\dots,\zeta_l$ be a system of generators of $S$ as an $R$-module. Take $a\in (\ba S)^{\lceil t\pi ^\epsilon\rceil}$.
Then for almost all $p$ , $a_p\in \ba_p^{\lceil t p^{e_p}\rceil}S_p$. Therefore, there exist $a_{ip}\in \ba_p^{\lceil t p^{e_p}\rceil}$ such that
$a_p = \sum_i a_{ip}\zeta_{ip}$. Hence, $a=\sum_i (\ulim a_{ip})\zeta_i\in\ba^{\lceil t\pi^\epsilon\rceil}S_\infty$.
\end{proof}
\begin{lem}
	Let $R$ be a $\Q$-Gorenstein normal local domain essentially of finite type over $\C$. Let $r$ be the minimum positive integer such that $rK_R$ is a Cartier divisor and let $S=\bigoplus_{i=0}^{r-1}\omega_R^{(i)}$ be a canonical cover of $R$. Then $S_\infty\cong R_\infty\otimes_R S$.
\end{lem}
\begin{proof}
	Let $A$ be a normal $\Q$-Gorenstein domain of finite type over $\C$ such that $\omega_A^{(r)}$ is free and $\p\in\Spec A$ such that $R\cong A_\p$. Let $B\cong \bigoplus_{i=0}^{r-1}\omega_A^{(i)}$ be a canonical cover such that $B_\p\cong S$. Take $Z$ such that there exist models $(Z,A_Z)$ and $(Z,B_Z)$ as in Theorem \ref{affine model}. Since $B$ is finite $A$-module, there exist $m, n\in \N$ and an exact sequence
	\begin{equation}\label{exact sequence a}
		A^m\to A^n \to B\to 0.
	\end{equation}
	Enlarging these models if necessary, we have an exact sequence
	\begin{equation}
		A_Z^m\to A_Z^n \to B_Z \to 0.
	\end{equation}
	Tensoring this exact sequence with $(A_p)_{\p_p}$, we have an exact sequence
	\begin{equation}
		R_p^m\to R_p^n \to S_p\to 0
	\end{equation}
	for almost all $p$.
	Taking the ultraproducts, we have
	\begin{equation}
		R_\infty^m\to R_\infty^n \to S_\infty \to 0.
	\end{equation}
	Tensoring the exact sequence (\ref{exact sequence a}) with $R_\infty$, we have an exact sequence
	\begin{equation}
		R_\infty^m\to R_\infty^n \to S\otimes_R R_\infty \to 0.
	\end{equation}
	Hence, we have $S_\infty \cong S\otimes_R R_\infty$.
\end{proof}
\begin{thm}\label{canonical-cover}
Let $(R, \ba^t)$ be a $\Q$-Gorenstein pair. Assume that $S$ is the canonical cover of $R$. Then
\begin{equation}
\tau_{\rm u}(R,\ba^t)=\tau_{\rm u}(S,(\ba S)^t)\cap R.
\end{equation}
\end{thm}
\begin{proof}
First, we prove $\tau_{\rm u}(R,\ba^t)\supseteq\tau_{\rm u}(S,(\ba S)^t)\cap R$. Since $R\hookrightarrow S$ is \'{e}tale in codimension one, the support of the kernel and the cokernel of $\omega_R\otimes_R S\to \omega_S$ has codimension at least two. Hence, we have
\begin{equation}
E_R\otimes_R S\cong \operatorname{H}_{\n}^d(\omega_R\otimes_R S)\cong \operatorname{H}_{\n}^d(\omega_S)\cong E_S.
\end{equation}
Let $x\in\tau_{\rm u}(S,(\ba S)^t)\cap R$ and $\eta \in \cl{E_R}{\ba}$. Then there exists $c\in R^\circ$ such that for all $\epsilon \in {^*\N}$ and $a\in \ba^{\lceil t\pi^\epsilon \rceil}$ we have $F^\epsilon_* ca\otimes \eta=0$ in $F^\epsilon_* R\otimes_R E_R$. We have a commutative diagram
\begin{equation}
\xymatrix{
E_R \ar[r]\ar[d]_{\xi\mapsto F^\epsilon_* ca\otimes \xi} & E_S \ar[d]^{\xi\mapsto F^\epsilon_* ca\otimes \xi} \\
F^{\epsilon}_{*}R_\infty\otimes_R E_R \ar[r] & F^{\epsilon}_{*}S_\infty \otimes _S E_S
}.
\end{equation}
 Hence, $F^\epsilon_*ca\otimes \eta=0$ in $F^\epsilon_*S_\infty\otimes_S E_S$. By Lemma \ref{ideal finite extension}, we have $\eta\in \cl{E_S}{(\ba S)}$. Hence, $x\eta=0$ in $E_S$. Since $E_R\to E_S$ is injective, we have $x\eta=0$ in $E_R$. Therefore, we have $\tau_{\rm u}(R,\ba^t)\supseteq\tau_{\rm u}(S,(\ba S)^t)\cap R$ by the definition of ultra-test ideals. For the converse, take $\eta\in \cl{E_S}{(\ba S)}$ and $x\in \tau_{\rm u}(R,\ba^t)$. Suppose that $x\eta\neq 0$. Then there exists $y\in S^\circ$ such that $xy\eta$ is a generator of the socle of $E_S$. By \cite[Lemma 2.3]{geometric-interpretation}, the socle of $E_S$ as an $S$-module is equal to the socle of $E_R$ as an $R$-module. Hence, we have $xy\eta \in E_R$. Let $r\in \N$ be the minimum positive integer such that $rK_R$ is a Cartier divisor. Since $E_S\cong E_R\otimes S$ and $S$ is a $\Z/r\Z$-graded $R$-algebra, $E_S$ is a $\Z/r\Z$-graded $R$-module. Let $(y\eta)_0$ be the 0-th homogeneous component of $y\eta$. Since $\deg x=0$, we have $x(y\eta)_0=xy\eta\neq 0$. Since $S_\infty \cong R_\infty \otimes S$, $S_\infty$ is a $\Z/r\Z$-graded $R_\infty$-algebra. Let $\pi:S\to R$ and $\pi_\infty:S_\infty\to R_\infty$ be projections to the 0-th homogeneous components. Let $s\in S$ be a homogeneous element with $\deg s=i$ and $\epsilon \in {^*\N}$. Then $F^\epsilon(s)$ is a homogeneous element of degree $i \pi^{\epsilon}\mod r$, where $j\equiv i \pi^{\epsilon} \mod r$ means that $j\equiv ip^{e_p}\mod r$ for almost all $p$. Since $p$ does not divide $r$ for almost all $p$, if $i\not\equiv 0 \mod r$, we have $i\pi^{\epsilon} \not\equiv 0 \mod r$. Hence, we have a commutative diagram
\[
	\xymatrix{
		S \ar[r] \ar[d]_{\pi}& F^\epsilon_* S_\infty \ar[d]^{F^\epsilon_*\pi_\infty} \\
		R \ar[r] & F^\epsilon_*R
	}.
\]
There exists $c\in S^\circ$ such that for all $\epsilon\in {^*\N}$ and $a\in \ba^{\lceil t\pi^\epsilon \rceil}$, $F^\epsilon_* ca\otimes \eta=0$ in $F^\epsilon_*S_\infty\otimes_S E_S$. Since $R\to S$ is finite, there exists $b\in R^\circ$ such that $b$ is a multiple of $c$. Consider a commutative diagram
\[
	\xymatrix{
		E_S \cong S\otimes E_R \ar[d]_-{\pi\otimes \id} \ar[rr]^-{\xi \mapsto F^\epsilon_* ba\otimes \xi} & & F^\epsilon S_\infty \otimes_R E_R\cong F^\epsilon_*S_\infty \otimes_S E_S \ar[d]^-{F^\epsilon_*\pi_\infty\otimes \id} \\
		E_R \ar[rr]^-{\xi \mapsto F^\epsilon_*ba\otimes \xi} & & F^\epsilon_* R_\infty\otimes_R E_R
	}.
\]
We have $F^\epsilon_*ba\otimes (y\eta)_0=0$ in $F^\epsilon_*R_\infty\otimes_R E_R$. Hence, $(y\eta)_0\in \cl{E_R}{\ba}$. Hence, we have $x(y\eta)_0=0$ by the definition of ultra-test ideals. This is a contradiction. Therefore, $x\eta=0$, which completes the proof.
\end{proof}
\begin{defn}[ultra-tight closure]
Let $(R,\ba^t)$ be a pair. We define the {\it ultra-tight closure} $\operatorname{cl}_{\rm u}^{\ba^t} (I)$ of $I$ as follows:
$x\in\operatorname{cl}_{\rm u}^{\ba^t} (I)$ if and only if there exists $c\in R^{\circ}$ such that for any 
\begin{equation}
\epsilon\in\NN, a\in\ba^{\lceil t\pi^\epsilon \rceil},
\end{equation} 
we have 
\begin{equation}
caF^\epsilon (x)\in F^\epsilon(I)R_\infty.
\end{equation}
\end{defn}
\begin{rem}\label{ultra-tight rem}
For $x\in R$, $x\in\operatorname{cl}_{\rm u}^{\ba^t}(I)$ if and only if $\overline{x}\in \cl{{R/I}}{\ba}$, where $\overline{x}$ is the equivalence class of $x$ in $R/I$. Therefore, $I:(\operatorname{cl}_{\rm u}^{\ba^t}(I))=\Ann_R \cl{{R/I}}{\ba}$.
\end{rem}
\begin{defn}[{\cite[Definition 5.2]{affine}}]
Let $(R,\ba^t)$ be a pair.
The {\it generic tight closure} $\operatorname{cl}_{\rm gen}^{\ba^t} (I)$ of $I$ is defined by
\begin{equation}
\operatorname{cl}_{\rm gen}^{\ba^t} (I)=(\ulim I_p^{*\ba^t})\cap R.
\end{equation}
\end{defn}
\begin{prop}\label{gen and ultra}
Let $(R,\ba^t)$ be a $\Q$-Gorenstein pair. Then $\operatorname{cl}_{gen}^{\ba^t} (I) \subseteq \operatorname{cl}_{u}^{\ba^t}(I)$.
\end{prop}
\begin{proof}
Take $x\in \operatorname{cl}_{\rm gen}^{\ba^t}$ and $c \in \mathcal{J}(R,\ba^t)$. For almost all $p$, for any $e_p\in\N$, $a_p\in\ba_p^{\lceil tp^{e_p}\rceil}$ $c_p a_p F^{e_p}(x)\in F^{e_p}(I_p)R_p$. Hence, for any $\epsilon\in\NN$ and for any $a\in\ba^{\lceil t\pi^\epsilon\rceil}$, we have $caF^\epsilon(x)\in F^\epsilon(I)R_\infty$. It follows that $x\in \operatorname{cl}_{\rm u}^{\ba^t}(I)$.
\end{proof}
\begin{thm}
Let $R$ be a normal quasi-Gorenstein local domain of characteristic $p>0$ and let $\ba$ be a non-zero ideal of $R$, $t\in\R_{>0}$. Assume that $I$ is an ideal of $R$ generated by a system of parameters. Then $I^{*\ba^t}=I:\tau(R,\ba^t)$.
\end{thm}
\begin{proof}
The proof is same as the Gorenstein case proved in \cite{Huneke}.
\end{proof}
\begin{thm}\label{test parameter}
Let $(R,\ba^t)$ be a $\Q$-Gorenstein pair.
\begin{equation}
\tau_{\rm u}(R,\ba^t)=\bigcap_{I} I:\operatorname{cl}_{\rm u}^{\ba^t}(I)\subseteq \bigcap_{I}I:\operatorname{cl}_{\rm gen}^{\ba^t}(I),
\end{equation}
where $I$ runs through all parameter ideals of $R$.
\end{thm}
\begin{proof}
By Corollary \ref{finitistic ultra} and Remark \ref{ultra-tight rem}, $\tau_{\rm u}(R,\ba^t)=\bigcap_{I} I:\operatorname{cl}_{\rm u}^{\ba^t}(I)$, where $I$ runs through all ideals of $R$. By Proposition \ref{gen and ultra}, the conclusion follows.
\end{proof}
\begin{thm}\label{ultra-test-mult-quasi-Goren}
Let $(R,\ba^t)$ be a quasi-Gorenstein pair. Then $\tau_{\rm u}(R,\ba^t)=\mathcal{J}(R,\ba^t)$.
\end{thm}
\begin{proof}
We will use similar arguments as in \cite[Theorem 6.2]{Puresubrings}.
Let $X\to Y:=\Spec R$ be a log resolution of $\ba$ such that $\ba\sO_X=\sO_X(-Z)$ is invertible, and set $W=f^{-1}(\m)$.
Considering approximations, we have the corresponding morphisms $X_p \to Y_p:=\Spec R_p$, invertible sheaves $\ba_p\sO_{X_p}=\sO_{X_p}(-Z_p)$ and closed fibers $W_p=f_p^{-1}(\m)$.
Then we have the following commutative diagram. (Note that $\sO_X(tZ):=\sO_X(\lfloor tZ\rfloor)$)
\begin{equation}
\xymatrix{
& \H_\m^d(R) \ar[d]^\gamma \ar[rd]^\delta & \\
\H^{d-1}(X, \sO_X(tZ)) \ar[r] \ar[d]& \H^{d-1}(X-W,\sO_X(tZ)|_{X-W}) \ar[r] \ar[d]^{u^{d-1}}& \H_W^d(\sO_X(tZ))\\
\UH^{d-1}(X,\sO_X(tZ))\ar[r]^-{\rho_\infty^{d-1}} & \UH^{d-1}(X-W,\sO_X(tZ)|_{X-W}) & \\
}
\end{equation}
Similarly, we have the following.
\begin{equation}
\xymatrix{
& \H_{\m_p}^d(R_p) \ar[rd]^{\delta_p} \ar[d]^{\gamma_p} & \\
\H^{d-1}(X_p,\sO_{X_p}(tZ_p)) \ar[r]_-{\rho_p^{d-1}} & \H^{d-1}(X_p-W_p, \sO_{X_p}(tZ_p)|_{X_p-W_p}) \ar[r] & \H_{W_p}^d(\sO_{X_p}(tZ_p))\\
}
\end{equation}
Assume $\eta \in \Ker \delta$. Then $u^{d-1}\circ \gamma(\eta)\in \operatorname{Im} \rho_{\infty}^{d-1}$. Therefore, $\gamma_p(\eta_p)\in\operatorname{Im}\rho_p^{d-1}$ for almost all $p$. Hence $\eta_p\in\Ker\delta_p$ for almost all $p$.
By the proof of \cite[Theorem 6.9]{HY}, 
\begin{equation}
0^{*\ba_p^t}_{\H_{\m_p}^d(R_p)} \supseteq \Ker \delta_p.
\end{equation}
Let $\eta=[\frac{a}{y}]\in \Ker \delta$, where $y$ is the product of a system of parameters ${\bf x}$.
Then $a_p \in ({\bf x}_p R_p)^{*\ba_p^t}$ for almost all $p$. It follows that $a\in \operatorname{cl}_{\rm gen}^{\ba^t} ({\bf x}R)$. If $b\in {\bf x}R:\operatorname{cl}_{\rm gen}^{\ba^t}({\bf x}R)$, then $b\eta =0$.
Therefore, 
\begin{equation}
\bigcap_n ({\bf x}^nR:\operatorname{cl}_{\rm gen}^{\ba^t}({\bf x}^nR)) \subseteq \Ann \Ker \delta =\mathcal{J}(R,\ba^t).
\end{equation}
By Theorem \ref{test parameter}, the conclusion follows.
\end{proof}
\begin{cor}\label{main1}
Let $(R,\ba^t)$ be a $\Q$-Gorenstein pair.
Then
\begin{equation}
\tau_{\rm u}(R,\ba^t)=\mathcal{J}(R,\ba^t).
\end{equation}
\end{cor}
\begin{proof}
It follows from Theorem \ref{canonical-cover}, Theorem \ref{ultra-test-mult-quasi-Goren}.
\end{proof}
The following theorem is a non-standard analytic analog of {\cite[Proposition 1.12]{HY}}.
\begin{thm}\label{pure-test}
Let $R, S$ be local domains essentially of finite type over $\C$, let $\ba$ be a non-zero principal ideal of $R$ and let $t\in\R_{>0}$. Assume $R\hookrightarrow S$ is a pure ring extension. Then 
\begin{equation}
\tau_{\rm u}(S,(\ba S)^t)\cap R\subseteq \tau_{\rm u}(R,\ba^t).
\end{equation}
\end{thm}
\begin{proof}
Let $M$ be a finite $R$-module.
For 
\begin{equation}
c\in R^\circ, \epsilon\in\NN, a\in R_\infty, K_{c,\epsilon,a}:=\Ker (M\to F^\epsilon_{*}R_\infty\otimes_R M; \eta\mapsto ca\otimes \eta).
\end{equation}
 Similarly, for 
\begin{equation}
c\in S^\circ, \epsilon\in\NN, a\in S_\infty, L_{c,\epsilon,a}=\Ker (S\otimes_R M\to F^\epsilon_{*}S_\infty\otimes_SS\otimes_R M;\eta\mapsto ca\otimes \eta).
\end{equation}
 If $c\in R^\circ, \epsilon\in\NN, a\in R_\infty$, we have an injection $K_{c,\epsilon,a}\hookrightarrow L_{c,\epsilon,a}$ via the injection $M\hookrightarrow S\otimes_R M$. Since $\ba$ is principal, $\ba^{\lceil t\pi^\epsilon\rceil}S_\infty =(\ba S)^{\lceil t\pi^\epsilon \rceil}$, we have 
\begin{equation}
\bigcap_{\epsilon\in\NN}\bigcap_{a\in\ba^{\lceil t\pi^\epsilon\rceil}}K_{c,\epsilon,a}\hookrightarrow\bigcap_{\epsilon\in \NN}\bigcap_{a\in (\ba S)^{\lceil t\pi^\epsilon\rceil}} L_{c,\epsilon,a}.
\end{equation}
Hence,
\begin{equation}
\cl{M}{\ba}=\bigcup_{c\in R^\circ}\bigcap_{\epsilon\in\NN}\bigcap_{a\in\ba^{\lceil t\pi^\epsilon\rceil}}K_{c,\epsilon,a}\hookrightarrow 
\bigcup_{c\in S^\circ}\bigcap_{\epsilon\in \NN}\bigcap_{a\in (\ba S)^{\lceil t\pi^\epsilon\rceil}} L_{c,\epsilon,a}=\cl{{S\otimes_R M}}{(\ba S)}.
\end{equation}
Therefore, 
\begin{equation}
\tau_{\rm u}(S,(\ba S)^t)\subseteq \Ann_S(\cl{{S\otimes_R M}}{(\ba S)})\cap R\subseteq \Ann_R(\cl{M}{\ba}).
\end{equation}
 By Corollary \ref{finitistic ultra},
\begin{equation}
\tau_{\rm u}(S,(\ba S)^t)\cap R\subseteq \bigcap_{M} \Ann_R(\cl{M}{\ba})=\tau_{\rm u}(R,\ba^t),
\end{equation}
where $M$ runs through all finitely generated $R$-modules.

\end{proof}
\begin{defn}[{\cite[Definition 9.3.60]{La}}]
Assume that $X$ is a normal variety over $\C$, $\Delta$ is a $\Q$-Weil divisor such that $K_X+\Delta$ is a $\Q$-Cartier divisor, $\ba$ is a non-zero ideal sheaf of $\sO_X$ and $t$ is a positive real number. Take a log resolution $\mu:Y\to X$ of $\ba$ such that $\ba \sO_Y=\sO_Y(-Z)$. Then the {\it multiplier ideal} $\mathcal{J}((X,\Delta),\ba^t)$ of $((X,\Delta),\ba^t)$ is defined by
\begin{equation}
\mathcal{J}((X,\Delta),\ba^t)=\mu_*\sO_Y(\lceil K_Y-\mu^*(K_X+\Delta)-tZ\rceil).
\end{equation} 
\end{defn}
\begin{defn}[\cite{de-Fernex-Hacon}]\label{non-Q-Gorenstein}
If $S$ is a (not necessarily $\Q$-Gorenstein) normal local domain essentially of finite type over $\C$, $\ba$ is a non-zero ideal of $S$ and $t\in\R_{>0}$, the {\it multiplier ideal} of $(S,\ba^t)$ is defined as the following.
\begin{equation}
\mathcal{J}(S,\ba ^t):=\sum_{\Delta}\mathcal{J}((\Spec S,\Delta),\ba^t),
\end{equation}
where $\Delta$ runs through all effective $\Q$-divisors such that $K_X+\Delta$ is $\Q$-Cartier.
If this multiplier ideal is trivial, then $(S,\ba^t)$ is said to be {\it Kawamata log terminal in the sense of de Fernex-Hacon}.
\end{defn}

\begin{cor}\label{main2}
Let $R, S$ be normal local domains essentially of finite type over $\C$, let $\ba$ be a non-zero principal ideal of $R$ and let $t\in\R_{>0}$. Assume $R$ is $\Q$-Gorenstein and $R\hookrightarrow S$ is a pure ring extension. Then $\mathcal{J}(S,(\ba S)^t)\cap R\subseteq \mathcal{J}(R,\ba^t)$. In particular if $S$ is Kawamata log terminal in the sense of de Fernex-Hacon, so does $R$. 
\end{cor}
\begin{proof}
Even if $S$ is not $\Q$-Gorenstein, $\mathcal{J}(S,\ba^t)\subseteq\tau_{\rm u}(S,\ba^t)$ by the same argument as in Theorem \ref{approx of mult} and Proposition \ref{mult<test}. Then the conclusion follows from Theorem \ref{pure-test}.
\end{proof}
\begin{rem}
We cannot use the usual reduction argument to conclude Corollary \ref{main2} because purity is not necessarily preserved under reduction modulo $p>0$. For example, you can take $t=2$, $m=n=3$ in Theorem 1.1 (a) in \cite{HJPS}.
\end{rem}
\begin{ques}
Can we prove Corollary \ref{main2} without using ultraproducts?
\end{ques}
\begin{rem}
	Let $R$ and $S$ be as in Corollary \ref{main2}. Quite recently Zhuang \cite{Zhuang} proved that if $S$ is Kawamata log terminal in the sense of de Fernex-Hacon, then so is $R$, without assuming that $R$ is $\Q$-Gorenstein. 
\end{rem}



\begin{data availability}
\textup{Data sharing not applicable to this article as no datasets were generated or analyzed during the current study.}
\end{data availability}

\end{document}